\documentclass[11pt]{article}

\usepackage{amsmath,amssymb,amsthm,graphicx,fixmath,tikz}

\usepackage{algorithmic}
\usepackage{algorithm}
\usepackage{setspace}
\usepackage{ifthen}

\def\Re{\mathbb R}
\newcommand*{\Positives}{\mathbb{N}}

\providecommand{\remove}[1]{}
\newcommand{\lemlab}[1]{\label{lemma:#1}}

\newcommand{\alglab}[1]{\label{algo:#1}}
\newcommand{\algref}[1]{Algorithm~\ref{algo:#1}}
\theoremstyle{plain}
\newtheorem{theorem}{Theorem}[section]
\newtheorem{lemma}[theorem]{Lemma}
\newtheorem{proposition}[theorem]{Proposition}

\newtheorem{corollary}[theorem]{Corollary}

\newtheorem{fact}[theorem]{Fact}
\theoremstyle{definition}
\newtheorem{definition}[theorem]{Definition}
\theoremstyle{remark}
\newtheorem{remark}[theorem]{Remark}
\newcommand{\R}{\mathcal{R}}

\newcommand{\E}{\mathcal{E}}
\renewcommand{\L}{\mathcal{L}}
\renewcommand{\P}{\mathcal{P}}
\newcommand{\calC}{\mathcal{C}}
\newcommand{\D}{\mathcal{D}}

\newcommand*{\Lfam}{\mathcal{L}}%
\newcommand{\chum}{{{ch_{\text{\textup{um}}}}}}
\newcommand{\chcf}{{{ch_{\text{\textup{cf}}}}}}
\newcommand{\chC}{{{ch_{\calC}}}}
\newcommand{\chium}{{{\chi_{\text{\textup{um}}}}}}
\newcommand{\chicf}{{{\chi_{\text{\textup{cf}}}}}}
\newcommand{\chiC}{{{\chi_{\calC}}}}
\newcommand{\cfch}{{{ch_{\text{\textup{cf}}}}}}

\newcommand{\cf}{{{\chi_{\text{\textup{cf}}}}}}

\newcommand{\ch}{{{ch}}}

\newcommand*{\HGp}[1]{H^{\text{\textup{path}}}_{#1}}%

\newcommand{\cardin}[1]{\lvert {#1} \rvert}
\newcommand{\card}[1]{\lvert {#1} \rvert}
\newcommand{\UComp}{{\mathcal{U}}}

\begin{document}

\title{The potential to improve the choice: \\
       list conflict-free coloring for geometric hypergraphs}

\author{%
Panagiotis Cheilaris\thanks{Center for Advanced Studies in Mathematics,
Mathematics department,
Ben-Gurion University, Be'er Sheva 84105, Israel. 
\texttt{panagiot@math.bgu.ac.il}
}
\and Shakhar Smorodinsky\thanks{Mathematics department, Ben-Gurion University,
Be'er Sheva 84105, Israel.
\texttt{shakhar@math.bgu.ac.il}
}
\and Marek Sulovsk{\'y}\thanks{Institute of Theoretical Computer Science,
ETH Zurich, 8092, Switzerland.
Research supported by the Swiss National Science Foundation (SNF project 200020-125027).
\texttt{smarek@inf.ethz.ch}
}%
}

\date{}
\maketitle


\begin{abstract}
Given a geometric hypergraph (or a range-space) $H=(V,\cal E)$, a coloring of its vertices is
said to be conflict-free if for every hyperedge $S \in \cal E$
there is at least one vertex in $S$ whose color is distinct from
the colors of all other vertices in $S$. The study of this notion
is motivated by frequency assignment problems in wireless
networks. We study the list-coloring (or choice)
version of this notion. In this version,
each vertex is associated with a set of (admissible) colors and it is allowed to be colored
only with colors from its set. List coloring arises naturally in the
context of wireless networks.

Our main result is a list coloring algorithm based on a new
potential method. The algorithm produces a stronger unique-maximum
coloring, in which colors are positive integers and
the maximum color in every hyperedge occurs uniquely.
As a corollary, we provide asymptotically sharp bounds on the size of the lists required to assure the existence of such
unique-maximum colorings
for many geometric hypergraphs (e.g., discs or pseudo-discs in the plane or points with respect to discs). Moreover, we provide an algorithm, such that, given a family of lists with the appropriate sizes,
computes such a coloring from these lists.
\end{abstract}


\section{Introduction and preliminaries}
Before introducing our results, let us start with several
definitions and notations that will be used throughout the paper.

\begin{definition}
Let $H=(V,\E)$ be a hypergraph and let $C$ be a coloring
$C \colon V \rightarrow \Positives$:
\begin{itemize}
\item We say that $C$ is a {\em proper coloring} if for every hyperedge $S \in
\E$ with $\card{S}\geq 2$ there exist two vertices $u,v \in S$
such that $C(u) \neq C(v)$. That is, every hyperedge with
at least two vertices is non-monochromatic.

\item We say that $C$ is a {\em conflict-free coloring} (cf-coloring in short) if for
every hyperedge $S \in \E$ there exists a color $i \in \Positives$
such that $\card{S \cap C^{-1}(i) }=1$. That is, every
hyperedge $S \in \E$ contains some vertex whose color is unique
in $S$.

\item We say that $C$ is a {\em unique-maximum coloring} (um-coloring in short) if for
every hyperedge $S \in \E$, $\card{S \cap C^{-1}(\max_{v \in S}C(v)) }=1$.
That is, in every
hyperedge $S \in \E$ the maximum color in $S$ is unique in $S$.

\end{itemize}
\end{definition}

We denote by $\chi(H)$, $\chicf(H)$, $\chium(H)$
the minimum integer $k$ for which $H$ admits
a proper, a conflict-free, a unique-maximum coloring, respectively,
with a total of $k$ colors. Obviously, every um-coloring of $H$ is a
cf-coloring of $H$ which is also a proper coloring of $H$,
but the converse is not necessarily true. Thus, we have:
$
\chi(H) \leq \chicf(H) \leq \chium(H)
$.
%

\paragraph{Conflict-free coloring.}
The study of cf-coloring was initiated in \cite{ELRS} and \cite{SmPHD} and was further studied in many settings (see, e.g.,
\cite{HURTADO,cf9,AS08,BCOScpc2010,cf7,CKS2009talg,cf5,HS02,Lev-TovP09,CFPT09,cf1}).
The study was initially motivated by its application to frequency
assignment for cellular networks. A cellular network consists of
two kinds of nodes: \emph{base stations} and \emph{mobile
clients}. Base stations have fixed positions and provide the
backbone of the network; they can be modeled, say, as discs in
the plane that represent the area covered by each
base station's antenna. Every
base station emits at a fixed frequency. If a client
wants to establish a link with a base station, it has to tune
itself to the base station's frequency. Clients, however, can be
in the range of many different base stations. To avoid
interference, the system must assign frequencies to base stations
in the following way: For any point $p$ in the plane
(representing a possible location of a client), there must be at least one
base station which covers $p$ and with a frequency that is not
used by any other base station covering $p$. Since frequencies
are limited and costly, a scheme that reuses frequencies, where
possible, is desirable. Let us formulate this in the language of
hypergraph coloring. Let $D$ be the set of discs representing the
antennas. We thus seek the minimum number of colors $k$ such
that one can assign each disc with one of the $k$ colors so that
in every point $p$ in the union of the discs in $D$, there is at
least one disc $d \in D$ that covers $p$ and whose color is
distinct from all the colors of other discs containing $p$. This
is equivalent to finding the cf-chromatic number of a certain
hypergraph $H=H(D)$ whose vertex set is $D$ and whose hyperedges
are defined by the Venn diagram of $D$.
Below, we give a formal definition for $H(D)$.

\paragraph{Geometric hypergraphs.}
Let $P$ be a set of $n$
points in the plane and let $\R$ be a family of regions in the
plane (such as all discs, all axis-parallel rectangles, etc.). We
denote by $H=H_{\R}(P)$ the hypergraph on the set $P$ whose
hyperedges are all subsets $P'$ that can be cut off from $P$ by a
region in $\R$. That is, all subsets $P'$ such that there exists
some region $r \in \R$ with $r\cap P = P'$. We refer to such a
hypergraph as the hypergraph {\em induced by $P$ with respect to
$\R$}.

For a finite family $\R$ of planar regions, we denote by $H(\R)$
the hypergraph whose vertex set is $\R$ and whose hyperedge set
is the family $\{\R_p \mid p \in \Re^2 \}$ where $\R_p \subseteq
\R$ is the subset of all regions in $\R$ that contain $p$. We
refer to such a hypergraph as the hypergraph {\em induced by
$\R$}.

Consider, for example, the (infinite) family $D$ of all discs in the
plane. In \cite{ELRS}, it was proved that
for any finite set $P$ of $n$ points, we have
$\cf(H_{D}(P)) = O(\log n)$. Similar questions can be asked
for other families of geometric hypergraphs where one is
interested in bounds on any of the chromatic numbers defined earlier.

\paragraph{Unique-maximum coloring.}
Most cf-coloring algorithms in the literature
produce unique-maximum colorings (which are stronger than
conflict-free colorings). The main reason for this approach is that
unique-maximum colorings have more structure than conflict-free
colorings, and thus it seems easier to argue about them in proofs.
One
interesting question is how can a non-unique-maximum conflict-free
coloring improve on a unique-maximum coloring, with respect to the
number of colors used and this line of research has been pursued
in \cite{ChKParxiv2010,ChT2010ciac}.

\paragraph{List coloring.}
Until now, research on cf-coloring was carried out under the assumption that we can use any color from some global set of colors.
The goal was to minimize the total number of colors used. In real
life, it makes sense to assume that each antenna in the wireless
network is further restricted to use a subset of the available
spectrum. This restriction might be local (depending, say, on the
physical location of the antenna). Hence, different antennas
may have different subsets of (admissible) frequencies available for them.
Thus, it makes sense to study the list version of conflict-free
coloring. That is, assume further that
each antenna $d \in D$ is associated with a subset $L_d$ of
frequencies. We want to assign
to each antenna $d$
a frequency that is taken from its allowed set $L_d$.
The following problem
thus arises:
\textsl{%
What is the minimum number $f=f(n)$ such that given
any set $D$ of $n$ antennas (represented as discs) and any family
of subsets of positive integers $\L=\{L_d \}_{d \in D}$ associated with
the antennas in $D$, the following holds: If each subset $L_d$ is
of cardinality $f$, then one can cf-color the hypergraph $H=H(D)$
from $\L$}.
In what follows, we give a formal definition of the
coloring model.

\begin{definition}
Let $H=(V,\E)$ be a hypergraph and let $\L = \{L_v\}_{v \in V}$
be a family of $\card{V}$ subsets of positive integers. We say
that $H$ \emph{admits a cf-coloring from} $\L$
if there exists a cf-coloring
$C \colon  V \rightarrow \Positives$ such that
$C(v) \in L_v$ for every $v \in V$.
Analogous definitions apply for the notions of a hypergraph $H$
\emph{admitting
a proper} or \emph{a um-coloring from} $\L$.
\end{definition}

\begin{definition}
We say that a hypergraph $H=(V,\E)$ is {\em $k$-cf-choosable}
if for every family
$\L=\{L_v\}_{v \in V}$ such that $\card{L_v}\geq k$ $\forall v \in
V$, $H$ admits a cf-coloring
from $\L$.
Analogous definitions apply for the notions of a hypergraph $H$ being
\emph{$k$-choosable} or \emph{$k$-um-choosable}.
\end{definition}

In this paper we are interested in the minimum number $k$ for
which a given geometric hypergraph is $k$-cf-choosable
(respectively, $k$-choosable, $k$-um-choosable).
We refer to this number as the {\em cf-choice
number}
(respectively, {\em choice number} and {\em um-choice number})
of $H$ and denote
it by $\cfch(H)$ (respectively, $ch(H)$ and $\chum(H)$).
Obviously, if the
cf-choice number
of $H$ is $k$
then $\cf(H) \leq k$, as one can cf-color
$H$ from $\L=\{L_v\}_{v \in V}$ where for every $v$ we
have $L_v = \{1,\ldots,k\}$
(the same can be said for proper and um colorings).
Thus,
\begin{equation}
\ch(H) \geq \chi(H), \quad
\cfch(H) \geq \cf(H), \quad
\chum(H) \geq \chium(H).
   \label{eq:choicechi}
\end{equation}
It is also easy to see that all of those parameters are upper-bounded by the number of vertices of the underlying hypergraph.

The study of list coloring for the special case of graphs, i.e.,
$2$-uniform hypergraphs, was initiated in
\cite{ERT1979,Vizing1976choice}. List proper coloring of
hypergraphs has been studied more recently, as well; see, e.g.,
\cite{KV2001chrandhyp}. We refer the reader to the survey of Alon
\cite{Alon93restrictedcolorings} for more on list coloring of
graphs.

\paragraph{Our results.}
In this paper we study the choice number, the cf-choice number, and the um-choice number of
hypergraphs.
We focus mainly on
geometric hypergraphs.
%
%

Our main result is an asymptotically tight
bound of $O(\log n)$ on the um-choice number of $H(\R)$ when $\R$
is a family of $n$ planar Jordan regions with linear union-complexity. In order to obtain the above result,
in Section~\ref{sec:potential},
we introduce a potential method for list um-coloring hypergraphs
that has also other applications.
In Section~\ref{sec:geom}, we apply the potential method in list um-coloring
several geometric hypergraphs of interest.
In Section~\ref{sec:separator}, we obtain an asymptotically tight upper bound on
the cf-choice number of hypergraphs consisting of the vertices of
a planar graph together with all subsets of vertices that form a
simple path in the graph
(see \cite{ChT2010ciac} for applications of this class
of hypergraphs); it is not possible to prove a similar upper bound
on the um-choice number and indeed we show that the um-choice
number of a hypergraph induced by paths of a planar graph can be
substantially higher.
In Section~\ref{sec:upperboundfewedges}, using the list coloring approach,
we prove tight upper bounds on the um-choice number in terms of
(a)~the number of hyperedges in the hypergraph or
(b)~the maximum degree of a vertex.
These results extend results of
\cite{CheilarisCUNYthesis2009,CFPT09}.
Moreover,
the list coloring approach allows us to provide a more concise proof.
In Section~\ref{sec:chcol}, we provide a general bound on
the cf-choice number of any hypergraph in terms of its
cf-chromatic number. We show that for any hypergraph $H$ (not
necessarily of a geometric nature) with $n$ vertices we have:
$
\cfch(H)\leq \cf(H) \cdot \ln n +1
$.
The proof of this fact uses a
probabilistic argument, which is an extension of a
probabilistic argument first given in \cite{ERT1979}. There, it
was proved that the choice-number of every bipartite graph with
$n$ vertices is $O(\log n)$. Our argument can be generalized
to a large natural class of colorings (i.e., not just conflict-free),
however, we note that such a bound is not possible for $\chum$.
Finally, in Section~\ref{sec:listnmgeom},
we study the (proper) choice number of
several geometric hypergraphs and show that many of the known
bounds for the proper coloring of the underlying hypergraphs hold
in the context of list-coloring as well.

\section{A potential method for list um-coloring}\label{sec:potential}
Let us start with a simple example of a hypergraph which can be viewed as induced by points on the line
with respect to all intervals.
Let $[n] = \{1,\dots,n\}$.
For $s \leq t$, $s, t \in [n]$, we define the {(discrete) interval}
$[s,t] = \{i \mid s \leq i \leq t\}$.
The \emph{discrete interval hypergraph} $H_n$ has vertex set $[n]$ and
hyperedge set
$\{  [s,t]  \mid {s \leq t}\text{, }{s, t \in [n]} \}$.
It is not difficult to prove that $\cf(H_n) = \lfloor \log_2 n
\rfloor + 1$ (see, e.g., \cite{ELRS,SmPHD}). Therefore,
from inequality~\eqref{eq:choicechi}, we
have the lower bound $\cfch(H_n) \geq \lfloor \log_2 n \rfloor +
1$.
As a warmup, we prove that the above lower bound is tight:

\begin{proposition}\label{proposition:intervalchcf}
For every $n \geq 1$, $\cfch(H_n) \leq \lfloor \log_2 n \rfloor +1 $.
\end{proposition}

\begin{proof}
Assume, without loss of generality, that $n=2^{k+1} -1$. We will show
that $H_n$ is $k+1$ cf-choosable. The proof is by induction on
$k$. Let $\L = \{L_i\}_{i \in [n]}$, such that $\card{L_i}=k+1$,
for every $i$. Consider the median vertex $p = 2^k$.
Choose a color $x \in L_p$ and assign it to $p$. Remove $x$ from
all other lists (for lists containing $x$), i.e., consider $\L' =
\{L'_i\}_{i \in [n] \setminus p}$ where $L'_i =
L_i\setminus\{x\}$. Note that all lists in $\L'$ have size at
least $k$. The induction hypothesis is that we can cf-color any
set of points of size $2^k-1$ from lists of size $k$.
Indeed, the number of vertices smaller (respectively, larger)
than $p$ is exactly $2^k -1$. Thus, we
cf-color vertices smaller than $p$ and independently vertices
larger than $p$, both using colors from the lists of $\L'$.
Intervals that contain the median vertex $p$ also have the
conflict-free property, because
color $x$ is used only in $p$.
This completes the induction step and hence the
proof of the proposition.
\end{proof}

We now turn to the more difficult problem of bounding the
um-choice number. Even for the discrete interval hypergraph $H_n$,
a divide and conquer approach, along the lines of the proof of
Proposition~\ref{proposition:intervalchcf} is doomed to fail.
In such an approach, some vertex close to the median must be found,
a color must be assigned to it from its list, and this color must be
deleted from all other lists.
However, vertices close to the median might have only `low'
colors in their lists. Thus, while we are guaranteed
that a vertex close to the median is uniquely colored for intervals containing it,
such a unique color is not necessarily the maximal color for such intervals.

Instead, we use a new approach. Our approach provides a general framework for um-coloring hypergraphs
from lists. Moreover, when applied to many
geometric hypergraphs, it provides
asymptotically tight bounds for the um-choice number. 
First, 
we need the definitions of an independent set and of 
an induced sub-hypergraph.

\begin{definition}
Given a hypergraph $H=(V,\E)$, a subset $U \subseteq V$ is called
an \emph{independent set} in $H$ if it does not contain any hyperedge of cardinality at least $2$, i.e.,
for every $S \in \E$ with
$\card{S} \geq 2$, we have $S \not\subseteq U$.
Note that each color class of a proper coloring of $H$ is an independent set in $H$.
\end{definition}

\begin{definition}
Let $H=(V,\E)$ be a hypergraph. For a subset $V' \subseteq V$,
we refer to $H[V'] := (V',\{S\cap V' \mid S \in \E \})$ as the \emph{sub-hypergraph} of $H$ induced by $V'$.
\end{definition}

Below, we give an informal description of the approach, which is then
summarized in Algorithm~\ref{algo:umchoice_general}. 

We start by sorting the colors in the union of all lists in increasing order.
Let $c$ denote the minimum color.
Let $V^c \subseteq V$ denote the subset of vertices containing $c$ in their lists.
Note that $V^c$ might contain very few vertices, in fact, it might be that $\cardin{V^c} = 1$.
We simultaneously color
a suitable subset $U \subseteq V^c$ of vertices in $V^c$ with $c$.
We make sure that $U$
is independent in the hypergraph $H[V^c]$.
The exact way in which we choose $U$ is crucial to the performance of the algorithm and is discussed below.
Next, for the uncolored vertices in $V^c \setminus U$, we remove the color $c$ from their lists.
This is repeated for every color in the union $\bigcup_{v \in V}L_v$ in increasing order of the colors. The algorithm stops
when all vertices are colored. Notice that such an algorithm might run into a problem, when all colors
in the list of some vertex are removed before this vertex is colored. Later, we show that if we choose
the subset $U \subseteq V^c$ in a clever way and the lists are sufficiently large, then we avoid such a problem.

\begin{algorithm}[htb!]
\caption{UMColorGeneric($H$, $\L$): Unique-max color hypergraph
         $H = (V, \E)$ from family $\Lfam$}
\alglab{umchoice_general}
\begin{algorithmic}
  \WHILE{$V \neq \emptyset$}
  \STATE $c \leftarrow \min \bigcup_{v \in V} L_v$
  \COMMENT{$c$ is the minimum color in the union of the lists}
  \STATE $V^c \leftarrow \{v \in V \mid c \in L_v\}$
  \COMMENT{$V^c$ is the subset of remaining vertices containing $c$ in their lists}
  \STATE $U \leftarrow$ a ``good'' independent subset of the induced hypergraph $H[V^c]$
  \FOR{$x\in U$}
  \STATE $f(x) \leftarrow c$
  \COMMENT{color it with color $c$}
  \ENDFOR
  \FOR[for every uncolored vertex, remove $c$ from its list]
      {$v \in V^c\setminus U$}
  \STATE $L_v \leftarrow L_v \setminus \{c\}$
  \ENDFOR
  \STATE $V \leftarrow V \setminus U$
  \COMMENT{remove the colored vertices}
  \ENDWHILE
  \RETURN $f$
\end{algorithmic}
\end{algorithm}

As mentioned, 
Algorithm \ref{algo:umchoice_general} 
might cause some lists to run out of colors before coloring
all vertices. However, if this does not happen, we
prove that the algorithm produces a um-coloring.

\begin{lemma}
Provided that the lists associated with the vertices do not run out of colors during the execution of \algref{umchoice_general}, then the algorithm produces a um-coloring
from $\L$.
\end{lemma}

\begin{proof}
Consider any hyperedge $S \in \E$. Consider the last iteration $t$
of the while loop
during which some vertex of $S$ was colored. 
Let $c$ denote the color chosen in that iteration. Note that $c$ is a maximal color in $S$. We need to prove that it is also unique in $S$. Let $V^c$ denote the subset of uncolored vertices (until iteration $t$) containing $c$ in their lists and let $U \subseteq V^c$ denote the independent set in $H[V^c]$ chosen to be colored with $c$.
  Note that $S\cap V^c$ is a hyperedge in $H[V^c]$. We need to show that $\cardin{S \cap V^c} =1$. Indeed, assume to the contrary that $\cardin{S \cap V^c} \geq 2$. Then, since $S\cap V^c$ is a hyperedge in $H[V^c]$ and $U$ is independent in $H[V^c]$, we must have $S\cap V^c \not\subseteq U$. Therefore, there must be a vertex $v \in (S\cap V^c) \setminus U$. This means that such a vertex $v\in S\cap V^c$ is not colored in iteration $t$. Hence, it is colored in a later iteration, a contradiction.
  \end{proof}

The key ingredient, which will determine the necessary size of the
lists of $\L$, is the particular choice of the independent set in the
above algorithm. We assume that the hypergraph $H=(V,\E)$
is hereditarily $k$-colorable
for some fixed positive integer $k$. That is, for
every subset $V' \subseteq V$, the sub-hypergraph $H[V']$ induced by $V'$ admits a proper $k$-coloring.
This is the case in many geometric hypergraphs.
For example,
hypergraphs induced by planar discs, or pseudo-discs, or, more generally,
hypergraphs induced by regions having linear union complexity
have such
a hereditary colorability property for some small constant $k$
(see Section~\ref{sec:geom} for details).
We must also put some condition on the size of the lists
in family $\Lfam = \{L_v\}_{v \in V}$. 
With some hindsight, we require 
\[ \sum_{v \in V} \lambda^{-\card{L_v}} < 1, \]
where $\lambda := \frac{k}{k-1}$.
We are ready to state the main theorem.

\begin{theorem}\label{thm:um_lists}
Let $H = (V, \E)$ be a hypergraph which is hereditarily
$k$-colorable and set $\lambda := \frac{k}{k-1}$. 
Let $\L = \{L_v\}_{v \in V}$, such that
$\sum_{v \in V} \lambda^{-\card{L_v}} < 1$.
Then, $H$ admits a unique-maximum coloring from $\L$.
\end{theorem}

\begin{proof}
We refine Algorithm~\ref{algo:umchoice_general}, by showing
how to choose a good independent set.

A natural choice of a good independent set would of course be the
largest one. Unfortunately, such a naive approach does not work.
Instead, we consider a potential function on subsets of uncolored vertices
and
we choose the independent set with the highest potential.
%
For an uncolored vertex $v \in V$,
let $r_t(v)$ denote the number of colors remaining
in the list of $v$ in the beginning of iteration $t$ of the algorithm.
Obviously, the value of $r_t(v)$ depends on the particular run
of the algorithm. 
For a subset of uncolored vertices $X \subseteq V$ 
in the beginning of iteration $t$, let
$P_t(X) := \sum_{v \in X}\lambda^{-r_t(v)}$.
We define the potential in the beginning of iteration $t$ to be
$P_t := P_t(V_t)$, 
where $V_t$ denotes the subset of all uncolored vertices 
in the beginning of iteration $t$.
Notice that the value of the potential in the beginning of the algorithm 
(i.e., in the first iteration) is
$P_1 =\sum_{v \in V} \lambda^{-\card{L_v}} < 1$.

Our goal is to show that, with the right choice of 
the independent set in each iteration, 
we can make sure that for any iteration $t$ and every vertex 
$v \in V_t$ the inequality $r_t(v) > 0$ 
holds.
In order to achieve this, we will show that,
with the right choice of the subset of vertices colored in each iteration,
the potential function $P_t$ is non-increasing in $t$.
This will imply that for any iteration $t$ 
and every uncolored vertex $v \in V_t$ 
 we have:
$$\lambda^{-r_t(v)} \leq P_t \leq P_1 < 1 $$
and hence $r_t(v) > 0$, as required.

Assume that the potential function is non-increasing up to iteration $t$.
Let $P_t$ be the value of the potential function in the beginning 
of iteration $t$
and let $c$ be the color associated with iteration $t$.
Recall that $V_t$ denotes the set of uncolored vertices that are 
considered in iteration $t$,
and $V^c \subseteq V_t$
denotes the subset of uncolored vertices that
contain the color $c$ in their lists.
Put $P' = P_t(V_t\setminus V^c)$
and $P'' = P_t(V^c)$. Note that $P_t= P'+ P''$.
Let us describe 
how we find the independent set of vertices to be colored at iteration $t$.
First, we find an auxiliary proper coloring of
the hypergraph $H[V^c]$ with $k$ colors.
Consider the color class $U$ which
has the largest potential $P_t(U)$.
Since the vertices in $V^c$ are partitioned into 
at most $k$ independent subsets $U_1, \ldots, U_k$
and $P''=\sum_{i=1}^k P_t(U_i)$,
then by the pigeon-hole principle there is an index $j$ for which
$P_t(U_j) \geq {P''}/{k}$.
We choose $U=U_j$ as the independent set to be colored at iteration $t$.
Notice that, in this case, the value $r_{t+1}(v) = r_{t}(v) - 1$ for 
every vertex
$v \in V^c\setminus U$,
and all vertices in $U$ are colored.
For vertices in $V_t \setminus V^c$, there is no change in the size of their lists.
Thus, the value $P_{t+1}$ of the potential function at the end of iteration $t$ (and in the beginning of iteration $t+1$) is
$P_{t+1} \leq P' + \lambda(1-\frac{1}{k})P''$.
Since 
$\lambda = \frac{k}{k-1}$, 
we have that $P_{t+1} \leq P' + P'' = P_t$, as required.
\end{proof}

We demonstrate the choice of the independent set according 
to the proof of Theorem~\ref{thm:um_lists}
in Algorithm~\ref{algo:umchoice_weights}, which is a refinement
of Algorithm~\ref{algo:umchoice_general}.

\begin{algorithm}
\caption{UMColor($H$, $\L$): Unique-max color the hypergraph
  $H = (V, \E)$  from $\Lfam$}
\alglab{umchoice_weights}
\begin{algorithmic}
  \REQUIRE $H$: a hereditarily $k$-colorable hypergraph
  \STATE $\lambda := \frac{k}{k-1}$
  \FOR{$v \in V$}
     \STATE $r(v) \gets \card{L_v}$
  \ENDFOR
  \WHILE{$V \neq \emptyset$}
  \STATE $c \leftarrow \min \bigcup_{v \in V} L_v$
  \COMMENT{$c$ is the minimum color in the union of the lists}
  \STATE $V^c \leftarrow \{v \in V \mid c \in \L_v\}$
  \STATE compute a proper coloring of $H[V^c]$
         with at most $k$ colors and
	 color classes $U_1, \dots,U_k$
  \STATE $U \gets \text{a color class among $U_1, \dots,U_k$ with
    $\max_{i \in \{1,\dots,k\}}\sum_{v\in U_i}\lambda^{-r(v)}$}$
  \FOR{$x\in U$}
  \STATE $f(x) \leftarrow c$
  \ENDFOR
  \FOR{$v \in V^c \setminus U$}
  \STATE $L_v \leftarrow L_v \setminus \{c\}$
  \COMMENT{remove the color $c$ from all lists of uncolored vertices in $V^c$}
  \STATE $r(v) \leftarrow r(v)-1$
  \COMMENT{update the number of remaining colors from the list of $v$}
  \ENDFOR
  \STATE $V \leftarrow V \setminus U$
  \COMMENT{remove the colored vertices}
  \ENDWHILE
  \RETURN $f$
\end{algorithmic}
\end{algorithm}

%

%

It is evident that 
Theorem~\ref{thm:um_lists} has an algorithmic version:

\begin{corollary}
  Let $H = (V,\E)$ be a hypergraph on $n$ vertices, which is hereditarily
  $k$-colorable and let $\L=\{L_v\}_{v \in V}$,
  such that $\sum_{v \in V} \lambda^{-\card{L_v}} < 1$.
  Assume that we have an efficient algorithm for
  $k$-proper-coloring every induced subhypergraph of $H$. 
  Then, we also have an efficient
  algorithm for unique-maximum coloring $H$ from $\L$.
\end{corollary}

We also show that the conditions of Theorem~\ref{thm:um_lists}
are, in a sense, best possible.
Remember the {discrete interval hypergraph} $H_n$ defined in the start 
of this section, with vertex set $V = [n]$.
It is easy to see that $H_n$ is hereditarily 2-proper-colorable,
i.e., $k=2$ and $\lambda=2$ in the notation of Theorem~\ref{thm:um_lists}.

\begin{theorem}
Given are $n$ positive integers
$x_1$, $x_2$, \ldots, $x_n$, such that 
$x_1 \geq x_2 \geq \dots \geq x_n$ and 
$\sum_{i \in [n]} 2^{-x_i} \geq 1$.
Then, there exists a family 
$\Lfam = \{L_i\}_{i \in [n]}$ 
(of sets of colors)
such that $\card{L_i} = x_i$ for every $i \in [n]$ 
and 
the discrete interval hypergraph $H_n$ does not admit 
a unique-maximum coloring from $\Lfam$.
\end{theorem}
\begin{proof}
We will prove by induction on $n$ that for the family 
$\Lfam = \{L_i\}_{i \in [n]}$ 
with 
$$L_i = \{c \mid x_1+1-x_i \leq c \leq x_1\},$$
for every $i \in [n]$,
$H_n$ does not admit a um-coloring from $\Lfam$.
For $n=1$, the above statement trivially holds.

For $n>1$, 
assume for the sake of contradiction that $H_n$ admits 
a um-coloring $C$ from $\Lfam$.
Without loss of generality, the 
maximum color in $C$ is $x_1$ and it occurs in exactly 
one vertex of $H_n$, say $m$, where $m \in [n]$.

Consider the following two discrete interval subhypergraphs of $H_n$:
$$ 
   \text{$H'$ induced by $[1,m-1]$} 
   \quad \text{and} \quad
   \text{$H''$ induced by $[m+1, n]$}, $$
and the families 
$$
 \Lfam'  = \{L'_i\}_{i \in [1,m-1]} = 
           \{L_i \setminus \{x_1\}\}_{i \in [1,m-1]}
   \quad \text{and} \quad 
 \Lfam'' = \{L''_i\}_{i \in [m+1,n]} = 
           \{L_i \setminus \{x_1\}\}_{i \in [m+1,n]}.
$$
Restricting coloring $C$ to each one of $H'$, $H''$  
shows that $H'$ admits a um-coloring from $\Lfam'$ and
$H''$ admits a um-coloring from $\Lfam''$.
By applying the inductive hypothesis for the smaller 
hypergraphs $H'$ and $H''$, we get
$$\sum_{1 \leq i < m} 2^{-\card{L'_i}} < 1 \quad \text{and} \quad
  \sum_{m < i \leq n} 2^{-\card{L''_i}} < 1.$$

We write 
\begin{equation}
\sum_{1 \leq i \leq n} 2^{-\card{L_i}} = P' + 2^{-\card{L_m}} + P'',
\label{eq:rewritesum}
\end{equation}
where $P' = \sum_{1 \leq i < m} 2^{-\card{L_i}} $
and 
$P'' = \sum_{m < i \leq n} 2^{-\card{L_i}} $.
If $P'' \geq 1/2$, then 
$$\sum_{m < i \leq n} 2^{-\card{L''_i}} = 
 \sum_{m < i \leq n} 2^{-(\card{L_i}-1)} = 2 P'' \geq 1,$$
which is a contradiction. Hence, $P'' < 1/2$.
Moreover, $P'' \leq 2^{-1} - 2^{-\card{L_m}}$, because 
$P''$ is a sum of multiples of $2^{-\card{L_m}}$
(remember that $\card{L_m} \geq \card{L_i}$ for $i \in [m+1,n]$).
Finally, from 
$\sum_{1 \leq i \leq n} 2^{-\card{L_i}} \geq 1$ 
and \eqref{eq:rewritesum}, we get
$$P' \geq 1 - P'' - 2^{-\card{L_m}} \geq 
         1 - 2^{-1} + 2^{-\card{L_m}} - 2^{-\card{L_m}} = 1/2,$$
which implies
$$\sum_{1 \leq i < m} 2^{-\card{L'_i}} = 
 \sum_{1 \leq i < m} 2^{-(\card{L_i}-1)} = 2 P' \geq 1,$$
which is a contradiction.
\end{proof}

\section{Geometric hypergraphs}
\label{sec:geom}

Consider a hypergraph $H=(V,\E)$ with $n$ vertices and a family 
$\Lfam = \{L_v\}_{v\in V}$, such that
for every $v \in V$,
$\card{L_v} > \log_{\lambda}{n} $. 
Then, $\sum_{v \in V} \lambda^{-\card{L_v}} < 1$ and thus we have the
following special case of Theorem~\ref{thm:um_lists}.

\begin{theorem}\label{thm:lambdachum}
Let $H = (V, \E)$ be a hypergraph which is hereditarily
$k$-colorable and set $\lambda := \frac{k}{k-1}$. 
Let $\L = \{L_v\}_{v \in V}$, such that
$\card{L_v} > \log_{\lambda}{n} $, for every $v \in V$.
Then, $H$ admits a unique-maximum coloring from $\L$.
\end{theorem}

As a corollary of Theorem~\ref{thm:lambdachum} we obtain asymptotically optimal bounds on the um-choice number (hence, also on the cf-choice number)
of many geometric hypergraphs.


\begin{corollary}
Let $C$ be some absolute constant. Let $\R$ be a (possibly infinite) family of simple planar Jordan regions such that, for any $n$ and any subset $\R' \subseteq \R$ of $n$ regions, the union
complexity of $\R'$ is bounded by $Cn$.
Let $H = H(\R')$ be a hypergraph induced by a subset $\R' \subseteq \R$ of $n$ such
regions.
 Then \[\chum(H) = O(\log n).\]
\end{corollary}
This follows from the fact that such a hypergraph has chromatic number $O(1)$ \cite{smoro} combined with Theorem~\ref{thm:lambdachum}.

\begin{corollary}
  Let $\D$ denote the (infinite) family of all planar discs.

   (i) Let $P$ be a set of $n$ points in the plane and let $H = H_{\D}(P)$ be the hypergraph induced by $P$ (with respect to $\D$). Then \[\chum(H) \leq \log_{4/3}n + 1 \approx 2.41 \log_2 n + 1.\]

   (ii) Let $\D' \subseteq \D$ be a set of $n$ discs. Then \[\chum(H(\D')) \leq \log_{4/3}n + 1.\]
\end{corollary}

This follows from combining the fact that such hypergraphs are hereditary $4$-colorable \cite{ELRS,smoro} together with Theorem~\ref{thm:lambdachum}.

\begin{corollary}
  Let $H$ be a hypergraph of $n$ points in $\Re^2$ with respect to
  halfplanes.
  Then \[\chum(H) \leq \log_{3/2}n + 1 \approx 1.71 \log_2 n + 1.\]
\end{corollary}

This is a consequence of the hypergraph being hereditarily 
3-colorable, except
when there are 4 points of which only 3 are at the convex hull
(see, e.g., \cite{BCOScpc2010,Keszegh07}).

\begin{corollary}
  Let $H$ be a hypergraph of $n$ points in $\Re$ with respect to
  intervals.
  Then \[\chum(H) \leq \log_{2}n + 1 .\]
\end{corollary}

The above hypergraph is isomorphic to the discrete interval 
hypergraph $H_n$, which is hereditarily 
2-colorable, as we have mentioned before.

%
%

\section{List cf-coloring of planar graphs with respect to paths}
\label{sec:separator}

Given a simple graph $G=(V,E)$, consider the hypergraph
\[
\HGp{G} =
 (V,
  \{S \mid \text{$S$ is the vertex set of a simple path in $G$}\}
 ).
\]
A cf (respectively, um) coloring of $\HGp{G}$ is called
a cf (respectively, um) coloring of $G$ \emph{with respect to paths}.
Unique-maximum
coloring of a graph $G$ with respect to paths is known in the literature as \emph{vertex ranking} or \emph{ordered coloring}. See e.g., \cite{DKKM94,orderedcoloring}.

\begin{theorem}\label{thm:planargraphschcf}
Let $G$ be a planar graph with $n$ vertices.
Then $\cfch(\HGp{G}) = O(\sqrt{n})$.
\end{theorem}
\begin{proof}
The proof is constructive. Given a
planar graph $G$ on $n$ vertices together with a
family $\L = \{L_v\}_{v \in V}$ of sets of
size $c\sqrt{n}$ where $c$ is some absolute constant to be
revealed later, we produce a cf-coloring $C$ of $G$ with
respect to paths with colors from $\L$.

The algorithm is recursive.
By the
Lipton-Tarjan separator theorem \cite{LT} and in particular by the
version of the separator theorem from \cite{Djidjev1982},
there exists a
partition of the vertex set
$V=R\cup B \cup S$ such that $\max(\card{R},\card{B})
\leq 2n/3$ and $\card{S} \leq \sqrt{6n}$ and such that
there is no edge connecting a vertex in $R$ with a vertex in $B$.
Moreover, this partition can be computed efficiently.

We color all vertices in
$S$ with distinct colors. This can be done greedily as follows:
Arbitrarily order the vertices in $S$ and for each vertex $v$ in
this order choose a color from $L_v$ to assign to $v$ which is
distinct from all colors assigned to previous vertices in $S$. This is
possible if $\card{L_v} = c\sqrt{n} \geq  \sqrt{6n} \geq
\card{S}$. Next, for each vertex $u \in R \cup B$ modify the lists
$\{L_u\}_{u \in R \cup B}$ by erasing all colors used for $S$,
namely put $\L' = \{\ L_u \setminus \{C(v) \mid v \in S \} \}_{u
\in R \cup B}$. We recursively color $G[B]$ and $G[R]$ from $\L'$. Note
that the colors assigned to vertices in $R \cup B$ are distinct
from all colors used for $S$. Note also that if this coloring is
indeed a valid cf-coloring of $G$ from $\L$ then the function
$f(n)$ defined to be the maximum cf-choice number for a planar graph on
$n$ vertices satisfies the following recursive inequality:
$$
f(n) \leq \sqrt{{6}}\sqrt{\smash[b]{n}} + f(2n/3) \leq
\sum_{i=0}^{\infty} \sqrt{6\cdot\left(\frac{2}{3}\right)^i n} =
\frac{\sqrt{6}\sqrt{n}}{1-\sqrt{2/3}} \approx
13.3485\sqrt{n}
$$
Thus, we have $f(n) \leq c\sqrt{n}$, for $c \approx 13.3485$,
as claimed.
\end{proof}

\begin{remark}
The upper bound $O(\sqrt{n})$ is asymptotically tight, since for
the $\sqrt{n}\times\sqrt{n}$ grid graph $G_{\sqrt{n}}$, it was
proved in \cite{ChT2010ciac} that $\cf(\HGp{G_{\sqrt{n}}}) =
\Omega(\sqrt{n})$ and thus, from inequality~\eqref{eq:choicechi},
also $\cfch(\HGp{G_{\sqrt{n}}}) = \Omega(\sqrt{n})$.
\end{remark}

It is easily seen that an analog of Theorem~\ref{thm:planargraphschcf} 
for $\chum$ does not hold.
For example, consider the \emph{star graph} on $n > 2$
vertices $K_{1,n-1}$.
It is easy to check that
$\chicf(\HGp{K_{1,n-1}}) = \chcf(\HGp{K_{1,n-1}}) = \chium(\HGp{K_{1,n-1}}) = 2$.
However, consider a family $\Lfam$ of lists as follows:
Associate the $n-1$ leaves of the star with the same list of colors
and associate the other vertex with a list of colors which are all lower
than the colors appearing in the lists of the leaves.
In a unique-maximum coloring with respect to paths from $\Lfam$,
no two leaves can get the same color, because then we do not have the unique maximum
property for the path that connects these two leaves and hence,
$\chum(\HGp{K_{1,n-1}}) \geq n-1$.

\section{List um-coloring hypergraphs with few edges}%
\label{sec:upperboundfewedges}

In this section we extend upper bounds on $\chicf$ from
\cite{CheilarisCUNYthesis2009,CFPT09},
making them hold also for $\chum$.
In other words, we extend the results in two ways making
them hold for choice instead of chromatic number and
for unique-maximum colorings instead of conflict-free colorings.
Moreover, we provide a more concise proof.
In order to state the results, we need the following definition.

\begin{definition}
For every hypergraph $H$, define $s(H)$ to be the minimum positive
integer $s$ such that $\card{E(H)} \leq s(s-1)/2$.
\end{definition}

\begin{fact}\label{fact:diffedges}
For two hypergraphs $H$ and $H'$,
if $\card{\E(H)} \geq \card{\E(H')}$ and $s(H)=s(H') > 1$, then
$$\card{\E(H)} - \card{\E(H')} < s(H) - 1.$$
\end{fact}
\begin{proof}
We have $(s(H)-1)(s(H)-2)/2 < \card{\E(H')} \leq \card{\E(H)} \leq s(H)(s(H)-1)/2$,
which implies $\card{\E(H)} - \card{\E(H')} < s(H) - 1$.
\end{proof}

\begin{theorem}\label{thm:familydegs}
Let $H=(V,\E)$ be a hypergraph  and let $\Lfam = \{L_v\}_{v \in V}$ be a family of lists.
If for every $v \in V$  $\card{L_v} \geq \min(\deg_H(v) + 1,s(H))$,
then
$H$ admits a unique-maximum coloring from $\Lfam$.
\end{theorem}
\begin{proof}
Notice that if $s(H)=1$, the hypergraph has no hyperedge and
thus if $\card{L_v} \geq 1$ for every $v \in V$,
then $H$ admits a unique-maximum coloring from $\Lfam$.
The proof is by induction on $\card{V}$. If $H$ has one vertex $v$,
then $\deg_H(v)=0$ and $s(H)=1$. Hence, if $\card{L_v} \geq 1$,
then $H$ admits a unique-maximum coloring from $\Lfam$.

If $\card{V} > 1$ and $s(H)>1$, 
consider the maximum color occurring in the union of all
lists, that is, $c = \max \bigcup_{v \in V} L_v$.
Among these vertices which have $c$ in their list,
choose the vertex $v$ with maximum degree in the hypergraph.
Consider the subset of hyperedges $\E_v \subseteq \E$ that contain $v$.
Put 
$H' = (V' , \E'$),
where
$V' = V \setminus \{v\}$,
$\E' = \E \setminus \E_v$,
and define
$\Lfam' = \{L'_{u}\}_{u \in V'}$ such that
\[
L'_{u} =
  \begin{cases}
     L_{u} \setminus \{c\}  &  \text{if $u \in \bigcup_{S \in \E_v}S$,}   \\
     L_{u}              &  \text{if $u \notin \bigcup_{S \in \E_v}S$.}
  \end{cases}
\]

In order to apply the inductive hypothesis on $H'$,
we prove that for every $u \in V'$,
\begin{equation}\label{eq:condition}
\card{L'_u} \geq  \min ( \deg_{H'}(u)+1, s(H') ) ,
\end{equation}

If $\card{L_u} = \card{L'_u}$, that is, when $u \notin \bigcup_{S \in \E_v}S$ or
$c \notin L_u$, then condition~\eqref{eq:condition} holds.
Also, if $c \in L_u$, $u \in \bigcup_{S \in \E_v}S$, and $s(H') < s(H)$, then
condition~\eqref{eq:condition} holds, since
$\deg_{H'}(u) < \deg_H(u)$.
If $c \in L_u$, $u \in \bigcup_{S \in \E_v}S$, and $s(H')=s(H)$, then
Fact~\ref{fact:diffedges}
implies
\[s(H) > \card{\E} - \card{\E'} + 1 = \deg_{H}(v)+1
  \geq \deg_{H}(u)+1 . \]
This follows from the fact that
$\deg_{H}(v)=\card{\E} - \card{\E'}$ and
$\deg_{H}(v) \geq \deg_{H}(u)$.
Since $s(H')=s(H)$ and $\deg_{H}(u) \geq \deg_{H'}(u)+1$,
we also have
\[s(H') > \deg_{H'}(u)+1.\]
As a result,
\begin{align*}
  \card{L'_u} & = \card{L_u}-1 \geq
  \min(\deg_{H}(u)+1,s(H))-1 =
  \deg_{H}(u)+1 - 1  \\
  & =
  \deg_{H'}(u)+1 =
  \min(\deg_{H'}(u)+1,s(H')).
\end{align*}

Finally,
by the inductive hypothesis, $H'$ admits a um-coloring from
$\Lfam'$. Extend this coloring by coloring $v$ with $c$ to get a
um-coloring of $H$.
\end{proof}

\begin{corollary}
For every hypergraph $H$, $\chum(H) \leq \Delta(H) + 1$.
\end{corollary}

\begin{corollary}
For every hypergraph $H$, $\chum(H) \leq s(H)$.
\end{corollary}



\section{A connection between choosability and colorability
in general hypergraphs}
\label{sec:chcol}

\begin{definition}
We call $C'$ a \emph{refinement} of a coloring $C$ if
$C(x) \neq C(y)$ implies $C'(x) \neq C'(y)$.
A class $\calC$ of colorings is said to have \emph{the refinement property}
if every refinement of a coloring in the class is also in the class.
\end{definition}

The class of conflict-free colorings and the class of proper colorings
are examples of classes which have the refinement property. On the
other hand, the class of unique-maximum colorings does not have this
property.

For a class $\calC$ of colorings we can define as usual the notions of chromatic
number $\chiC$ and choice number $\chC$. 
Then, we can prove the following theorem 
for classes with the refinement property.

\begin{theorem} \label{thm:general}
For every class of colorings $\calC$ that has the refinement property
and every hypergraph $H$ with $n$ vertices,
$\chC(H) \leq \chiC(H) \cdot \ln n +  1$.
\end{theorem}
\begin{proof}
If $k = \chiC(H)$,  there is a $\calC$-coloring $C$ of $H$ with colors
$\{1, \dots, k\}$,
which induces a partition of $V$
into $k$ classes: $V_1 \cup V_2 \cup \dots \cup V_k$.
Consider a family $\L = \{ L_{v} \}_{v \in V}$, such that
for every $v$,  $\card{L_v} = k^* > k \cdot \ln n$.
We wish to find a family $\L' = \{L'_{v}\}_{v \in V}$ with the
following properties:
\begin{enumerate}
\item For every $v \in V$, $L'_v \subseteq L_v$.
\item For every $v \in V$, $L'_v \neq \emptyset$.
\item For every $i \neq j$, if $v \in V_i$ and $u \in V_j$, then
      $L'_v \cap L'_u = \emptyset$.
\end{enumerate}
Obviously, if such a family $\L'$ exists, then there exists a
$\calC$-coloring from $\L'$: For each $v \in V$, pick a color $x \in
L'_v$ and assign it to $v$.

We create the family $\L'$ randomly as follows: For
each element in $\cup \L$, assign it uniformly at random to one
of the $k$ classes of the partition $V_1 \cup \dots \cup V_k$.
For every vertex $v \in V$, say with $v \in V_i$, we create
$L'_v$, by keeping only elements of $L_v$ that were assigned
through the above random process to $v$'s class, $V_i$.

The family $\L'$ obviously has properties~1 and~3.
We will prove that with positive probability it also has property~2.

For a fixed $v$, the probability that
$L'_v = \emptyset$
is at most
\[
\left(1 - \frac{1}{k}\right)^{k^*} \leq e^{-k^*/k} < e^{-\ln n} =
\frac{1}{n}
\]
and therefore,
using the union bound,
the probability that
for at least one
vertex $v$, $L'_v = \emptyset$,
is at most
\[n \left(1 - \frac{1}{k}\right)^{k^*} <  1.\]
Thus, there is at least one family $\L'$
where property~2 also holds,
as claimed.
\end{proof}

\begin{corollary}
For every hypergraph $H$,
$\chcf(H) \leq \chicf(H) \cdot \ln n +  1$.
\end{corollary}

\begin{corollary}
For every hypergraph $H$,
$\ch(H) \leq \chi(H) \cdot \ln n +  1$.
\end{corollary}

The argument in the proof of Theorem~\ref{thm:general}
is a generalization of an argument
first given in \cite{ERT1979}, proving that any bipartite graph
with $n$ vertices is $O(\log n)$-choosable (see
also~\cite{Alon1992probchoice}).

We can not have an analog of Theorem~\ref{thm:general}
for unique maximum colorings.
Again, as in the end of Section~\ref{sec:separator},
the counterexample is the hypergraph with respect to paths of the
star graph, $\HGp{K_{1,n-1}}$,  for which
$\chium(\HGp{K_{1,n-1}}) = 2$, whereas
$\chum(\HGp{K_{1,n-1}}) \geq n-1$.

\section{Choice number of geometric hypergraphs}%
\label{sec:listnmgeom}

In this section, we provide near-optimal upper bounds on the choice
number of several geometric hypergraphs. We need the following
definitions:

\begin{definition}
Let $\R$ be a family of $n$ simple Jordan regions in the plane.
The {\em union complexity} of $\R$ is the number of vertices
(i.e., intersection of boundaries of pairs of regions in $\R$)
that lie on the boundary $\partial \bigcup_{r \in \R} r$.
\end{definition}

\begin{definition}
Let $H=(V,\E)$ be a hypergraph. Let $G=(V,E)$ be the graph whose
edges are all hyperedges of $\E$ with cardinality two. We refer
to $G$ as the {\em Delaunay graph} of $H$.
\end{definition}

\begin{theorem} \label{thm:choice}
(i) Let $H$ be a hypergraph induced by a finite set of points in
the plane with respect to discs. Then $\ch(H) \leq 5$.

(ii) Let $D$ be a finite family of discs in the plane. Then
$\ch(H(D)) \leq 5$.

(iii) Let $\R$ be a set of $n$ regions and let $\UComp: \mathbb N
\rightarrow \mathbb N$
    be a function such that $\UComp(m)$ is the maximum complexity of any $k$ regions in $\R$ over all $k \leq m$,
    for $1 \leq m \leq n$.
    We assume that $ \frac{\UComp(m)}{m}$ is a non-decreasing function.
Then, $\ch(H(\R)) = O(\frac{\UComp(n)}{n})$.
\end{theorem}

\begin{proof}
(i) Consider the Delaunay graph $G=G(P)$ on $P$, where two points
$p$ and $q$ form an edge in $G$ if and only if there exists a disc
$d$ such that $d\cap P = \{p,q\}$. That is, there exists a disc
$d$ that cuts off $p$ and $q$ from $P$. The proof of (i) follows
easily from the following known facts:
\begin{enumerate}
\item Every disc containing
at least two points of $P$ must also contain a Delaunay edge
$\{p,q\} \in E(G)$. (see, e.g., \cite{ELRS}).

\item $G$ is planar (see, e.g., \cite{CGbook}).

\item
 Every planar-graph is $5$-choosable \cite{thomassen}.
\end{enumerate}

(ii) The proof of the second part follows from a reduction to
three dimensions from \cite{smoro} and Thomassen's result
\cite{thomassen}.

(iii) For the third part of the theorem, we need the following
lemma from \cite{smoro}:
\begin{lemma} \cite{smoro}
    Let $\R$ be a set of $n$ regions and let $\UComp: \mathbb N \rightarrow \mathbb N$
    be a function such that $\UComp(m)$ is the maximum complexity of any $k$ regions in $\R$ over all $k \leq m$,
    for $1 \leq m \leq n$. Then, the Delaunay graph $G$ of the hypergraph $H=H(\R)$ has
    a vertex with degree at most $c\frac{\UComp(n)}{n}$ where $c$ is some absolute constant.
    \lemlab{silly}
\end{lemma}

The proof is similar to the proof of \cite{smoro} of the fact that
$\chi(H(\R)) = O(\frac{\UComp(n)}{n})$. We prove that $\ch(H(\R))
\leq c \cdot \frac{\UComp(n)}{n} +1$. Let $\L=\{L_r\}_{r \in \R}$
be the sets associated with the regions of $\R$. The proof is by
induction on $n$. Let $r\in \R$ be a region with at most $c \cdot
\frac{\UComp(n)}{n}$ neighbors in $G$. By the induction
hypothesis, the hypergraph $H(\R \setminus \{r\})$ is $c \cdot
\frac{\UComp(n-1)}{n-1}+1 \leq c \cdot \frac{\UComp(n)}{n}
+1$-choosable (by our monotonicity assumption on
$\frac{\UComp(n)}{n}$). We need to choose a color (out of the $c
\cdot \frac{\UComp(n)}{n}+1$ colors that are available for us in
the set $L_r$) for $r$ such that the coloring of $\R$ is valid.
Obviously, points that are not covered by $r$ are not affected by
the coloring of $r$. Note also that any point $p \in r$ that is
contained in at least two regions of $\R \setminus r$ is not
affected by the color of $r$ since, by induction, the set of
regions in $\R \setminus \{r\}$ containing such points is
non-monochromatic. We thus only need to color $r$ with a color
that is different from the colors of all regions $r' \in \R
\setminus r$, for which there is a point $p$ that is contained
only in $r\cap r'$. However, by our choice of $r$, there are at
most $c \cdot \frac{\UComp(n)}{n}$ such regions. Thus, we can
assign to $r$ a color among the $c \cdot \frac{\UComp(n)}{n}+1$
colors available to us in $L_r$ and keep the coloring of $\R$
proper. This completes the inductive step.
\end{proof}

\begin{corollary}
Let $\P$ be a family of $n$ pseudo-discs (i.e., a family of simple
closed Jordan regions, such that the boundaries of any two of
them intersect at most twice). Then $\ch(H(\P)) = O(1)$.
\end{corollary}
The corollary follows immediately from the fact that such a family
$\P$ has linear union complexity \cite{KLPS}, combined with
Theorem~\ref{thm:choice}.

\section{Open problems}
We consider the following as an interesting problem left open here:
\begin{itemize}

\item
Let $H$ be a hypergraph induced by $n$ axis-parallel rectangles
in the plane. Is it true that $\ch(H) = O(\log n)$? It is known
that $\chi(H) = \Theta(\log n)$ \cite{PTrects,smoro}.
\end{itemize}

\subsubsection*{Acknowledgments.} 
We wish to thank Emo Welzl and Yelena Yuditsky
for helpful discussions concerning the problems studied in this
paper.


\bibliographystyle{plain}
\bibliography{references}

\end{document}